%% file: intord-dim.tex
\documentclass{amsart}

\usepackage{tikz,mathptmx}
\usepackage{pgfpages,tikz}
\usetikzlibrary{arrows}
\tikzset{vtx/.style={circle, inner sep=0pt, minimum size=0.22cm,fill=black}}

\DeclareMathAlphabet{\mathcal}{OMS}{cmsy}{m}{n}
\input{mitch}

\newcommand{\paren}[1]{\ensuremath{\left( #1 \right)}}
\renewcommand{\set}[1]{\ensuremath{\left\{ #1 \right\}}}

\renewcommand{\ileft}[2][{}]{\ensuremath{{\ell}_{#1}(#2)}}
\renewcommand{\iright}[2][{}]{\ensuremath{{r}_{#1}(#2)}}

\begin{document}
\title{Dimension of Restricted Classes of Interval Orders}
\author{Mitchel T.\ Keller}
\address{Morningside University\\1501 Morningside Avenue\\Sioux City, IA 51106}
\email{kellerm@morningside.edu}
\author{Ann N.\ Trenk}

\address{Department of Mathematics\\Wellesley College\\Wellesley, MA}
\email{atrenk@wellesley.edu}
\author{Stephen J.\ Young}
\address{Pacific Northwest National Laboratory\\Richland, WA 99352}
\email{stephen.young@pnnl.gov}
\date{\today}
\thanks{\textit{PNNL Information Release:} PNNL-SA-152654}
\thanks{
This work was supported by a grant from the Simons Foundation (\#426725, Ann Trenk). }
\maketitle
\begin{abstract}
  Rabinovitch showed in 1978 that the interval orders having a
  representation consisting of only closed unit intervals have order
  dimension at most \(3\). This article shows that the same dimension
  bound applies to two other classes of posets: those having a
  representation consisting of unit intervals (but with a mixture of
  open and closed intervals allowed) and those having a representation
  consisting of closed intervals with lengths in \(\set{0,1}\).
\end{abstract}

\section{Introduction and Background}

Recent years have seen a number of efforts to characterize the classes of interval orders between the unit interval orders (or semiorders) in which each interval has the same length and the full class of interval orders in which intervals of any nonnegative length are allowed. See, for instance, \cite{boyadzhiyskaSimpleProofCharacterizing2018,boyadzhiyskaIntervalOrdersTwo2019,isaakBoundedDiscreteRepresentations1993,joosCharacterizationMixedUnit2015,rautenbachUnitIntervalGraphs2013,shuchatUnitIntervalOrders2016,shuchatUnitMixedInterval2014}. To date, this work has focused on forbidden subposet characterizations and recognition algorithms. Similar questions involving interval \emph{graphs} have also been studied. See, for example, \cite{cerioliCountingIntervalLengths2011,fishburnClassesIntervalGraphs1985,koblerIntervalGraphRepresentation2015,peerIntervalGraphsSide1995,peerRealizingIntervalGraphs1997}. This paper begins to address how these classes interact with the order dimension of these posets. For example, Rabinovitch showed in 1978 in \cite{rabinovitch:semiorder-dim} that the dimension of a unit interval order is at most three. On the other hand, a series of papers have shown that for every positive integer \(d\), there exists an interval order of dimension at least \(d\) and developed a thorough understanding of the rate of growth of dimension in terms of typical poset parameters. (See section~\ref{sec:conclusion} for references.) However, the growth rate of the maximum dimension of an interval order having a representation using at most \(r\) distinct interval lengths has not been carefully studied. For the classes of interval orders studied in this paper, we show that the dimension is at most 3, which raises the compelling question of finding the ``smallest'' (in any of several reasonable senses) interval order of dimension \(4\).

\subsection{Special Classes of Interval Orders}

Following Trotter's survey article \cite{trotter:intords}, we consider posets to be reflexive. An \emph{interval order} is a poset \(P\) for which each element \(x\)
of the ground set of \(P\) can be assigned an interval
\([\ileft{x},\iright{x}]\) so that for distinct \(x\) and \(y\), \(x < y\) in \(P\) if and only if
\(\iright{x} < \ileft{y}\). That is, \(x < y\) in \(P\) if and only if
the interval assigned to  \(x\) lies completely to the left of the interval assigned to
\(y\).  It follows that  two elements are incomparable in \(P\) if and only if their
intervals intersect. Such an assignment of a collection of intervals is
called an \emph{interval representation} or just \emph{representation}
of \(P\). Interval orders were shown by Fishburn in
\cite{fishburn:intords-char} to be characterized by being the class of
posets that exclude \twoplustwo, the disjoint union of two chains on
two elements. General background on interval orders can be found in
the monograph of Golumbic and Trenk \cite{golumbicToleranceGraphs2004} or Trotter's survey article \cite{trotter:intords}.

A \emph{semiorder} or \emph{unit interval order} is an interval order
having a representation in which all intervals have the same
(typically unit) length. Scott and Suppes showed in
\cite{scott:semiorders} that unit interval orders can be characterized as the
posets that exclude \twoplustwo{} and \oneplusthree, where \oneplusthree{} is  the disjoint
union of an isolated point and a chain on three elements. While
interval representations are typically assumed to consist of closed,
bounded intervals, for finite interval orders, it is equivalent to
restrict to open, bounded intervals. Working from the perspective of graph theory, 
 Rautenbach and
Szwarcfiter showed in \cite{rautenbachUnitIntervalGraphs2013} that
allowing both open and closed intervals (all bounded) in a
representation does not expand the class of posets beyond the interval
orders. In addition, they give a forbidden graph characterization of the class of interval graphs having a representation
in which all intervals have unit length with both open and closed intervals allowed. 
Indeed, the following result follows immediately from the construction provided in part (i) of the proof of \cite[Proposition 1]{rautenbachUnitIntervalGraphs2013}.
\begin{lem}\label{lem:rep}
    Let $P$ be a poset having a representation consisting of any combination of open, closed, or half-open intervals, then $P$ is an interval order that can be represented by a collection of closed intervals.
\end{lem}

In \cite{shuchatUnitIntervalOrders2016}, Shuchat et al.\ considered the analogous class of interval orders having 
a representation in which all intervals have unit length with both
open and closed intervals allowed.   They called these posets
\emph{unit OC interval orders} and gave both a forbidden subposet characterization and a polynomial time
recognition algorithm for the class. It is straightforward to see that
\oneplusthree{} \emph{is} a unit OC interval order, so the class of unit OC interval
orders is strictly larger than the class of unit interval orders.   In
\cite{boyadzhiyskaIntervalOrdersTwo2019}, Boyadzhiyska et al.\ studied
the interval orders having a representation in which each (closed)
interval has length \(0\) or \(1\). The authors use digraph methods
to provide a forbidden subposet characterization of this class of
posets. These two classes of posets are our objects of study in this
paper. Because we are able to prove our results using only the 
representations of these posets, we do not reproduce the forbidden
subposet characterizations from \cite{boyadzhiyskaIntervalOrdersTwo2019} and  \cite{shuchatUnitIntervalOrders2016} here.

In \cite{bogart:canonical-intords}, Bogart et al.\ first studied the
dimension of interval orders and showed that there are interval orders
of arbitrarily large dimension. They did so by using the canonical
interval order on \(n\) endpoints, which consists of all intervals
having endpoints in the set \(\set{1,2,\dots,n}\). Subsequently,
Rabinovitch showed in \cite{rabinovitch:semiorder-dim} that if \(P\)
is a unit interval order, then its dimension is at most \(3\). The independent
characterization of the \(3\)-irreducible posets (with respect to
dimension) by Kelly in \cite{kelly:3-dim-irr} and Trotter and Moore in
\cite{trotter:3-dim-irr} subsequently made it easy to show that the
dimension of a unit interval order that is not a total order is \(3\) if and
only if it contains one of three posets on seven points shown in Figure~\ref{fig:semi-forb}. 

\begin{figure}[h]
  \centering
  \begin{tikzpicture}
    \draw (2,1) --(1,0) -- (1,1) -- (1,2);
    \draw (0,0) -- (0,1) -- (1,0);
    \draw (2,0) -- (2,1) -- (0,0) -- (1,2) -- (2,0);
    \draw [fill] (0,0) circle [radius=0.1];
    \node[vtx, label={[left] {$a_1$}}] at (0,0) {};
    \node[vtx, label={[left] {$b_1$}}] at (0,1) {};
    \node[vtx, label={[below, label distance = -0.15cm] {$a_2$}}] at (1,0) {};
    \node[vtx, label={[below left] {$b_2$}}] at (1,1) {};
    \node[vtx, label={[above] {$c$}}] at (1,2) {};
    \node[vtx, label={[right] {$a_3$}}] at (2,0) {};
    \node[vtx, label={[right] {$b_3$}}] at (2,1) {};
    \node at (1,-1) {$\mbf{FX}_2$};
  \end{tikzpicture}\hspace*{0.1\linewidth}
  \begin{tikzpicture}
    \draw (2,0.5) -- (1,0) -- (1,1) -- (1,2) -- (2,1.5);
    \draw (1,0) -- (0,1.5) -- (0,0.5) -- (1,2);
    \node [vtx, label={[left] {$a_1$}}] at (0,0.5) {};
    \node [vtx, label={[left] {$a_2$}}] at (0,1.5) {};
    \node [vtx, label={[below right, label distance=-0.1cm] {$b_1$}}] at (1,0) {};
    \node [vtx, label={[right] {$b_2$}}] at (1,1) {};
    \node [vtx, label={[right] {$b_3$}}] at (1,2) {};
    \node [vtx, label={[right] {$c$}}] at (2,0.5) {};
    \node [vtx, label={[right] {$d$}}] at (2,1.5) {};
    \node at (1,-1) {$\mbf{H}_0$};
  \end{tikzpicture}\hspace*{0.1\linewidth}
  \begin{tikzpicture}
    \draw (0,0) -- (0,1) -- (0,2) -- (1,0.5) -- (1,1.5) -- (1,2.5) -- (2,1.5) -- (0,0) -- (1,1.5);
    \draw (0,1) -- (1,2.5);
    \node [vtx, label={[left] {$a_1$}}] at (0,0) {};
    \node [vtx, label={[left] {$a_2$}}] at (0,1) {};
    \node [vtx, label={[left] {$a_3$}}] at (0,2) {};
    \node [vtx, label={[below right, label distance=-0.1cm] {$b_1$}}] at (1,0.5) {};
    \node [vtx, label={[right] {$b_2$}}] at (1,1.5) {};
    \node [vtx, label={[right] {$b_3$}}] at (1,2.5) {};
    \node [vtx, label={[right] {$c$}}] at (2,1.5) {};
    \node at (1,-0.75) {$\mbf{G}_0$};    
  \end{tikzpicture}
  \caption{The three subposets that can force a unit interval order to have dimension $3$}
  \label{fig:semi-forb}
\end{figure}
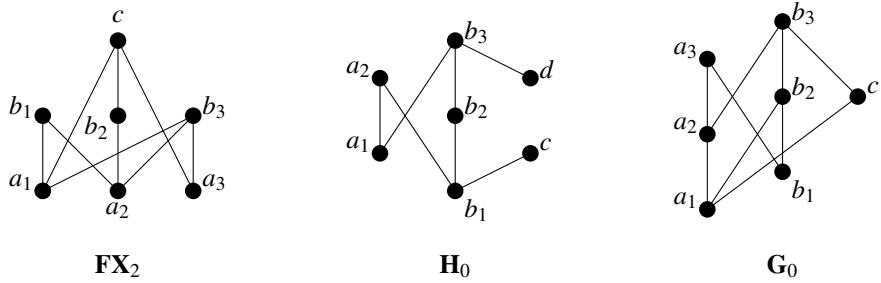

In this paper, we show that if \(P\) is a unit OC interval order or has a
representation consisting of intervals of lengths \(0\) and \(1\),
then the dimension of \(P\) is at most \(3\). While Rabinovitch's
argument for unit interval orders appears to rely on the lack of
\oneplusthree, we are able to give our proof by using the representation for
unit OC interval orders. The proof for interval orders that are
\(\set{0,1}\)-representable is more intricate and relies on other
techniques in dimension theory. The next subsection provides an
overview of the background in dimension theory that is required to
read the remainder of the paper. We conclude with some open questions
regarding the dimension of interval orders.


\subsection{Definitions}

The \emph{down set} of \(x\) in a poset \(P\) is \(D(x) = \set{z\in P\colon z <
  x\text{ in } P }\), and we write \(D[x] = D(x)\cup\set{x}\). The \emph{up
sets} \(U(x)\) and \(U[x]\) are defined dually.  We say that a poset
\(P\) has \emph{no duplicated holdings} provided that there do not
exist distinct points \(x,y\in P\) such that \(D(x) = D(y)\) and
\(U(x) = U(y)\). Some papers in this area refer to such posets as
\emph{twin-free}. 

If $P$ is a poset and $A,B\subseteq P$, we write $A < B$ to mean that
for all $a\in A$ and all $b\in B$, $a < b$. (And similarly for
$A > B$.) If $P$ is an interval order, this means that in any interval
representation of $P$, every interval assigned to a point of $A$ lies completely to the
left of every interval assigned to a point of $B$.

Given a poset \(P\), let \(\inc(P) = \set{(x,y)\in P\times P\colon x\text{ is incomparable to }y\text{ in }P}\). We call an ordered pair \((x,y)\in\inc(P)\) an \emph{incomparable pair}.  A   \emph{linear extension}  \(L\) of
\(P\)   is a total order on the ground set of \(P\)
such that if \(x < y\) in \(P\), then \(x < y\) in \(L\). A
\emph{realizer} of \(P\) is a set \(\mathcal{R}\) of linear extensions
of \(P\) such that \(x < y\) in \(P\) if and only if \(x < y\) in
\(L\) for all \(L\in\mathcal{R}\). Equivalently, \(\mathcal{R}\) is a
realizer for \(P\) if and only if for each incomparable pair \((x,y)\)
of \(P\), there exist \(L,L'\in\mathcal{R}\) such that \(x < y\) in
\(L\) and \(y<x\) in \(L'\). In this case, we say that \(L\) \emph{reverses} the incomparable pair \((y,x)\) and \(L'\)
reverses the incomparable pair \((x,y)\). Since \((x,y)\) being
an incomparable pair means that \((y,x)\) is also an incomparable
pair, it is sufficient to say that \(\mathcal{R}\) is a realizer
provided that each incomparable pair is reversed by a linear extension
in \(\mathcal{R}\). 

The \emph{dimension} of a poset \(P\) is the least positive integer
\(t\) such that there exists a realizer of \(P\) having cardinality
\(t\). Note that for posets of dimension at least \(2\), it is always
possible to assume that \(P\) has no duplicated holdings without
changing the dimension, since elements with the same up set and same
down set can be placed consecutively in  one linear extension of a realizer and consecutively in the reverse order in another to ensure that all incomparable pairs are
reversed.


A \emph{strict alternating cycle} of length \(k\) in a poset \(P\) is a sequence \(\set{(x_i,y_i)\colon
  1\leq i\leq k}\) such that
  \(x_i\leq y_{i+1}\) cyclically for \(i=1,2,\dots,k\) and
  \(x_i\) and \(y_j\) are incomparable when \(j\neq i+1\) (cyclically).
  In \cite{Tro-Moo-77}, Trotter and Moore showed that there is a
  linear extension reversing all incomparable pairs in a set \(S\subseteq P\times P\) if
  and only if \(S\) does not contain a strict alternating
  cycle. Figure~\ref{fig:strict-alt} shows a strict alternating cycle
  of length \(5\). The only comparabilities amongst the \(10\) points
  are those illustrated, but note that since our posets are reflexive,
  it is possible for the comparabilities to be equality. For instance,
  \(x_3=y_4\) is a possibility. 

  \begin{figure}[h]
    \centering
    \begin{tikzpicture}
      \foreach \x in {1,2,3,4,5} {\node[vtx, label={[below, label
          distance = -0.15cm] {$x_\x$}}] at (\x,0) {}; \node[vtx,
        label={[above] {$y_\x$}}] at (\x,2) {}; }

      \foreach \i in {1,2,...,4} { \pgfmathsetmacro{\t}{1+\i} \draw
        (\i,0) -- (\t,2); }
      \draw (5,0) -- (1,2);
    \end{tikzpicture}
    
    \caption{A strict alternating cycle}
    \label{fig:strict-alt}
  \end{figure}
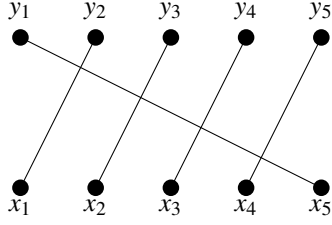
  
When \(P\) is an interval order, a strict alternating cycle takes a
particularly restrictive form, since in an interval order, a strict alternating cycle \(S\) of
length at least \(2\) may only contain one strict comparability or
else \(S\) witnesses that \(P\) contains \twoplustwo. We note for its
usefulness later that a strict alternating cycle of length \(2\) in an
interval order that contains a strict comparability must therefore
be \(\set{(x_1,y_1),(x_2,x_1)}\) with and \(x_2 < y_1\) and \(x_1\) incomparable to \(y_1\)
and \(x_2\). 

We now establish a useful consequence of the preceding observation. This originally appears in Rabinovitch's thesis \cite{rabinovitchDimensionTheorySemiordersIntervalOrders1973}, but can also be found in Trotter's monograph \cite[p. 196]{trotter:dimbook}. 

\begin{lem}\label{lem:rev-sets}
    If \(P\) is an interval order and \(A\) and \(B\) are disjoint subsets of \(P\), then there exists a linear extension \(L\) of \(P\) with \(a > b\) in \(L\) for all \(a\in A\) and \(b\in B\) for which \(a\) and \(b\) are incomparable in \(P\).
\end{lem}

\begin{proof}
    Let \(S = (A\times B)\cap\inc(Q)\). Because \(A\) and \(B\) are disjoint and the only strict alternating cycles \(C\) in the interval order \(P\) have an element that appears as both the first coordinate and a second coordinate of distinct pairs in \(C\), we know that \(S\) contains no strict alternating cycles. Thus, there is a linear extension \(L\) of \(P\) reversing all the pairs in \(S\). Such a linear extension has the desired property.
\end{proof}

\section{Dimension of Unit OC Interval Orders and Length
  \(\{0,1\}\)-Representable Interval Orders}
  
  Our first main result concerns the dimension of unit OC interval orders.  Since unit OC interval orders are not necessarily unit interval orders, in the proof we build a related unit interval order $Q$. We do this in order to take advantage of the ideas behind Rabinovitch's proof that unit interval orders have dimension at most \(3\).  

\begin{thm}\label{thm:unit-OC-dim}
  If $P$ is a unit OC interval order, then $\dim(P)\leq 3$.
\end{thm}

\begin{proof}
    Let \(P\) be a unit OC interval order and fix a representation $\mathcal{I}$
  of $P$ in which every interval has unit length.  Let \(Q\) be the
  unit interval order formed from \(P\) by making all of the intervals in
  \(\mathcal{I}\) open. Thus, in \(Q\) there may be pairs of points
  \(x,y\) such that \(x\) and \(y\) are comparable in \(Q\) and
  incomparable in \(P\). However, all comparabilities in \(P\) remain
  comparabilities in \(Q\).  This means that any linear extension of $Q$ is also a linear extension of $P$.
  We now partition \(Q\) into antichains
  \(A_1,A_2,\dots,A_t\) by successively removing the minimal elements
  of \(Q\). Since incomparabilities in $Q$ are also incomparabilities in $P$, each $A_i$ is also an antichain in $P$.

  We claim that the
  only incomparabilities in \(P\) are within an \(A_i\) or between
  consecutive antichains.   First we show that for any \(a\in A_i\) and \(b\in A_{i+2}\), \(a< b\) in
  \(P\). To prove this, assume for a contradiction that \(a\) and
  \(b\) are incomparable in \(P\). Hence, \(\ileft{b}\leq
  \iright{a}\). Since \(b\in A_{i+2}\), there exists \(x\in A_{i+1}\) with
  \(x < b\) in \(Q\). Similarly, there exists \(y\in A_i\) with \(y <
  x\) in \(Q\). Since \(y < x < b\) in \(Q\) and all intervals are open,
  \[\iright{y}\leq \ileft{x} = \iright{x} - 1\leq \ileft{b} - 1\leq
    \iright{a}-1 = \ileft{a}.\]
  If any of the inequalities above is strict, then \(\iright{y} <
  \ileft{a}\) and \(y < a\) in \(Q\), which contradicts that \(a\) and
  \(y\) are both in the antichain \(A_i\). Thus, equality must hold
  throughout and \(\ileft{x} = \ileft{a}\), meaning that \(x\) and
  \(a\) have the same interval in our representation of
  \(Q\). However, this would mean that \(x\) and \(a\) should be in
  the same antichain of our partition instead of having \(x\in
  A_{i+1}\) and \(a\in A_i\). Now for \(j > i+2\), \(c\in A_j\)
  requires that \(c\) is greater than some element of \(A_{i+2}\), and
  thus \(c\) is greater than all elements of \(A_i\). Therefore, the
  only incomparabilities in \(P\) are within an \(A_i\) or between
  consecutive antichains.

   We now form a witness \(\mathcal{R}=\set{L_1,L_2,L_3}\) to
  demonstrate \(\dim(P)\leq 3\).  The total orders $L_1$ and $L_2$ will be linear extensions of $P$, while $L_3$ will be a linear extension of both $P$ and $Q$. Let \(A = \bigcup_i A_{2i}\) and \(B = \bigcup_i A_{2i+1}\). 
  As $P$ is an interval order by Lemma~\ref{lem:rep} and $A$ and $B$ are disjoint, we may apply Lemma~\ref{lem:rev-sets}.  In particular, we define $L_1$ as the linear extension of $P$ formed by applying Lemma~\ref{lem:rev-sets} to place $A > B$ and similarly define $L_2$ to place $B > A$. Thus, we have that if $x$ and $y$ are incomparable in $P$ with $x \in A$ and $y \in B$, then $x > y$ in $L_1$ and $y > x$ in $L_2$. 
  For \(L_3\),
  we first order the elements by the antichains to which they belong
  \(A_1, A_2, \dots, A_t\), and when ordering the elements of \(A_i\),
  we place them in the dual order to their order in \(L_1\). Now
  \(\mathcal{R}\) is a realizer of \(P\) because we know all
  incomparabilities are either within an antichain \(A_i\), in which
  case the incomparable pairs appear in opposite orders in \(L_1\) and
  \(L_3\), or between consecutive antichains, in which case the
  incomparable pairs appear in opposite orders in \(L_1\) and \(L_2\).
\end{proof}

Notice that the argument that proves Theorem~\ref{thm:unit-OC-dim}
works equally well if we allow posets \(P\) having representations
including half-open intervals \((a,a+1]\) and \([b,b+1)\), since the construction of \(Q\) still preserves all comparabilities of \(P\) and Lemma~\ref{lem:rep} ensures that $P$ is an interval order so that Lemma~\ref{lem:rev-sets} can be applied. 
When unit half-open intervals are allowed, the resulting class of posets is larger than the unit OC interval orders.  For example, the poset in Figure~\ref{fig:non_unitOC} is not a unit OC interval order.
There is a \oneplusthree{} on \(y\) and \(\set{x_1,x_2, x_3}\). This forces \(y\) and \(x_2\) to be represented by intervals with the same endpoints, \(y\) to be represented by a closed interval, \(x_2\) to be represented by an open interval, and \(x_3\) to be represented by an interval that is closed on the left. Similarly, there is a \oneplusthree{} on \(y\) and \(\set{x_1, x_2, z}\), so \(z\) must be represented by an interval having the same endpoints as the interval of \(x_3\). Furthermore, the interval of \(z\) must be closed on the left. We now try to add an interval representing \(x_4\). It must be greater than \(x_3\) but incomparable with \(z\). Since the intervals representing \(x_3\) and \(z\) have the same endpoints, the only way to make this happen is for \(x_3\) to be represented by an interval that is open on the right and \(z\) to be represented by an interval that is closed on the right. Thus \(x_3\) must be represented by a half-open interval. This leads to the unit mixed interval representation shown at the right in Figure~\ref{fig:non_unitOC}.
 Joos \cite{joosCharacterizationMixedUnit2015} and Shuchat et al.\ \cite{shuchatUnitMixedInterval2014} study this larger class from the perspective of graph theory and independently give a forbidden graph characterization of these \emph{unit mixed interval graphs}. 
 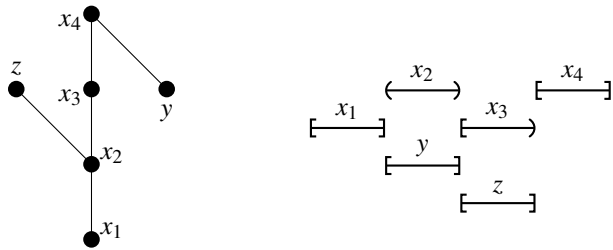
\begin{figure}[h]
  \centering
  \begin{tikzpicture}
    \draw (0,2) -- (1,1);
    \draw (1,0) -- (1,1) -- (1,2) -- (1,3) -- (2,2);
    \node[vtx, label={[right] {$x_1$}}] at (1,0) {};
    \node[vtx, label={[right] {$x_2$}}] at (1,1) {};
    \node[vtx, label={[below left] {$x_3$}}] at (1,2) {};
    \node[vtx, label={[below left] {$x_4$}}] at (1,3) {};
    \node[vtx, label={[above, label distance = -0.05cm] {$z$}}] at (0,2) {};
    \node[vtx, label={[below, label distance = -0.2cm] {$y$}}] at (2,2) {};
  \end{tikzpicture} \hspace*{0.1\linewidth}
  \begin{tikzpicture}
    \node [above] at (0,-0.6) {\phantom{\(x_1\)}};
    \draw[{[-]},thick] (0,1) -- (1,1);
    \node [above] at (0.5,1) {\(x_1\)};
    \draw[{[-)},thick] (2,1) -- (3,1);
    \node [above] at (2.5,1) {\(x_3\)};
    \draw[{(-)},thick] (1,1.5) -- (2,1.5);
    \node [above] at (1.5,1.5) {\(x_2\)};
    \draw[{[-]},thick] (3,1.5) -- (4,1.5);
    \node [above] at (3.5,1.5) {\(x_4\)};
    \draw[{[-]},thick] (1,.5) -- (2,.5);
    \node [above] at (1.5,.5) {\(y\)};
    \draw[{[-]},thick] (2,0) -- (3,0);
    \node [above] at (2.5,0) {\(z\)};
  \end{tikzpicture}
  \caption{A poset that is not a unit OC interval order and a unit mixed interval representation of it}
  \label{fig:non_unitOC}
\end{figure}
\begin{thm}\label{thm:01-dim}
  If $P$ is an interval order having a representation with each interval of length $0$ or $1$, then $\dim(P)\leq 3$.
\end{thm}
\begin{proof}
  Without loss of generality, we may assume that $P$ has no duplicated
  holdings, as duplicated holdings do not change dimension when posets have dimension 2 or more.
  Thus every
  representation of \(P\) has all distinct intervals. Let 
  $\mathcal{I} = \set{[\ileft{x},\iright{x}]\colon x\in P}$ be a representation of $P$ in
  which each interval has length $0$ or $1$. Let $Q$ be the subposet
  of $P$ consisting of the points represented by unit intervals in
  $\mathcal{I}$. Partition $Q$ into antichains $A_1,A_2,\dots,A_t$ by
  successively removing the minimal elements. For $i$ with $1 \le i \le t$, let
  $p_i = \min\set{\iright{x}\colon x\in A_i}$. That is, $p_i$ is the
  smallest right endpoint of an interval in $A_i$. We now partition
  the length $0$ intervals into $t+1$ sets based on their location
  relative to the $p_i$. For $i$ with $1 \le i \le t-1$, let $D_i$ be the length $0$
  intervals $[c,c]$ with $c\in (p_i,p_{i+1}]$. Let $D_0$ be the length
  $0$ intervals $[c,c]$ with $c\leq p_1$, and let $D_t$ be the length $0$
  intervals $[c,c]$ with $c>p_t$.

  Note that a length $1$ interval belongs to the antichain $A_i$ if and
  only if it contains $p_i$. Since each interval in $A_{i+1}$ is
  greater than an interval in \(A_i\), all intervals in \(A_{i+1}\)
  lie to the right of $p_i$. Therefore, $A_{i+2} > D_i$ for
  $i=0,\dots,t-2$. 
  Furthermore, $D_i > A_{i-1}$ for $i=2,\dots,t$
  since all intervals in $A_{i-1}$ must end before $p_{i}$. If \(x\in A_i\) and \(y\in A_{i+2}\), then there is  \(z\in A_{i+1}\) satisfying \(z < y\). Hence,
  \[\ileft{y} > \iright{z}\geq p_{i+1} > \iright{x}.\]
  This guarantees that $A_j > A_i$ if $j \geq
  i+2$.

  Hence, the incomparable pairs we must reverse involve pairs of
  points
  \begin{itemize}
  \item in $A_i$ and $A_{i+1}$,
  \item in $D_i$ and $A_i$,
  \item in $D_i$ and $A_{i+1}$, and
  \item both in $A_i$.
  \end{itemize}
  We will now define two disjoint sets of incomparable pairs and show that neither contains a strict alternating cycle. Let
  \begin{align*}
    S_1 &= \set{(x,y)\in \inc(P) \colon x\in A_{2i-1}, y\in A_{2i}\text{ or }x\in D_{2i-1}, y\in A_{2i}\text{ or }x\in A_{2i-1}, y\in D_{2i-1}}\\
    S_2 &= \set{(x,y)\in \inc(P) \colon x\in A_{2i}, y\in A_{2i+1}\text{ or }x\in D_{2i}, y\in A_{2i+1}\text{ or }x\in A_{2i}, y\in D_{2i}}.
  \end{align*}
  Suppose that $S_j$ contains a strict alternating cycle
$C$. By our earlier argument about strict alternating cycles in interval orders,
$C=\set{(d,y),(x,d)}$ is the entire strict alternating cycle
with $d$ represented by an interval of
length $0$. (If $d$ had length $1$, then $d$ would be in both $A_i$ and $A_{i+1}$ for some $i$, which is impossible.) By our definition of $S_1$ and $S_2$, this forces $x \in A_{k}$ and $y
\in A_{k+1}$ for some $k$. However, the definition of a strict alternating cycle implies that $d$ is incomparable with both $x$ and $y$ and that $x<y$. This is a contradiction since a length 0 interval cannot intersect two disjoint intervals.

  Since $S_j$ does not contain a strict alternating cycle, there is a linear extension $L_j$  of $P$ that reverses all the incomparable pairs of $S_j$ for \(j=1,2\). Form a third linear extension of $P$ with $D_0 < A_1 < D_1 < A_2 < \cdots < A_t < D_t$. Each $D_i$ is a chain, so the ordering of those points is fixed. For each $A_i$, order the points in the dual order to how they appear in $L_1$. 
  
  We now verify that $\set{L_1,L_2,L_3}$ realizes $P$ and therefore $\dim(P)\leq 3$. In particular, a pair $(a_i,a_{i+1}) \in \inc(P)\cap \paren{A_i \times A_{i+1}}$ is reversed in $L_1$ if $i$ is odd and in $L_2$ if $i$ is even. The pair $(a_{i+1},a_i)$ is reversed in $L_3$.  Similarly, the pairs $(a_i,d_i) \in \inc(P) \cap \paren{A_i \times D_i}$ and $(d_i,a_{i+1}) \in \inc(P)\cap\paren{D_i \times A_{i+1}}$ are reversed in $L_1$ if $i$ is odd and in $L_2$ if $i$ is even, while $(d_i,a_i)$ and $(a_{i+1},d_i)$ are reversed in $L_3$.  Finally all pairs $(a_i,a_i') \in \inc(P)\cap A_i^2$ are either reversed in $L_1$ or in $L_3$, where the elements of $A_i$ appear in the dual order as in $L_1$.
\end{proof}

Notice that the proof of Theorem~\ref{thm:01-dim} works equally well
if all intervals have length \(0\) and \(r\) for some fixed positive
real number \(r\), since the subposet consisting of the length \(r\) intervals is a unit interval order. While it seems possible to relax the 
requirement of the unit intervals being closed in
Theorem~\ref{thm:01-dim}, the modifications to the proof appear to be more
intricate than those required for Theorem~\ref{thm:unit-OC-dim}.

\section{Conclusion}
\label{sec:conclusion}
We conclude briefly with some related open questions inspired by this work. 
\begin{enumerate}
    \item Here we address the case of interval orders that can be represented using
  intervals of only length \(0\) and some positive length. What is the 
  bound on the dimension of interval orders having a representation
  consisting only of intervals of lengths \(r\) and \(s\) with \(r,s > 0\)?

\item More generally, what is the growth rate of the function \(f\) such that if \(P\) is an interval
  order having a representation using at most \(r\) different interval
  lengths, then \(\dim(P)\leq f(r)\)? 
  \end{enumerate}
  It is easy to see that the answer to the first question is at most 8 by partitioning the poset into two semiorders, $R$ and $S$, consisting of the length $r$ and length $s$ intervals.  The incomparabilities in $R$ and $S$ can each be reversed using 3 linear extensions, and two additional linear extensions suffice to reverse incomparabilities between $R$ and $S$.  This idea can be generalized \cite{trotterPersonalCommunication2020} to show that $f(r) \leq 3r + \binom{r}{2}$.  However, based on the work of F\"uredi et
  al.\ in \cite{furedi:intord-shift} showing  that the dimension of the canonical interval order with \(n\) different interval lengths is asymptotically \[\lg(\lg(n)) +
  \left(\frac{1}{2} + o(1)\right)\lg(\lg(\lg(n))),\] we expect that this bound on $f(r)$ is extremely loose and that $f(r)$ in fact grows extremely slowly, i.e., $O(\lg\lg(r))$.  It is worth noting that using improved estimates for the dimension  of the shift graph~\cite{hostenOrderDimensionComplete1999, kleitmanDedekindProblemNumber1975} the error term on the estimate of the dimension of the canonical interval order is actually at most 5~\cite{trotterPersonalCommunication2020}. 
  
  In \cite{bosekOnlineVersionRabinovitch2012}, Bosek et al.\ use marking functions to provide an alternative proof of Rabinovitch's theorem that the dimension of a unit interval order is at most \(3\). These methods may provide alternative (but likely no shorter) proofs of the theorems in this paper, but care would have to be taken because the marking function arguments require representations with distinct endpoints, which cannot be assured when allowing both open and and closed unit intervals. (For instance, \oneplusthree{} can only be represented as such an order by allowing repetition of endpoints.) 
  
  It would also be interesting to find the smallest \(n\) such that there exists an interval
  order on \(n\) points having dimension \(4\) and to determine the minimum number of interval lengths required to force the dimension to \(4\). To the best of the
  authors' knowledge, all arguments showing the existence of a
  4-dimensional interval order require a large poset with a relatively large number of interval lengths. An answer to this problem may be related to the first question above, were it possible to show that the best possible bound is \(4\) by giving a straightforward construction of a four-dimensional interval order that can be represented using intervals of two lengths.

\section*{Acknowledgment}

The authors would like to thank an anonymous referee who pointed out an error in the application of Lemma~\ref{lem:rev-sets} in the proof of Theorem~\ref{thm:unit-OC-dim} in an earlier version of this paper.

\bibliographystyle{acm}
\bibliography{zotero-abbrev}

\end{document}

%% file: mitch.tex
\usepackage{amsmath, amsthm, amssymb, amsfonts}
\usepackage{mathrsfs}

\newtheorem{thm}{Theorem}

\newtheorem{lem}[thm]{Lemma}

\newtheorem*{thm*}{Theorem}
\newtheorem*{cor*}{Corollary}
\newtheorem*{lem*}{Lemma}
\newtheorem*{prop*}{Proposition}
 
\newtheorem*{ex*}{Exercise} 

\theoremstyle{definition}

\newtheorem*{defn*}{Definition}

\newtheorem*{prob*}{Problem}

\theoremstyle{definition}

\theoremstyle{definition}

\newtheorem*{example*}{Example}

\theoremstyle{remark}

\newtheorem*{conj*}{Conjecture}

\newcommand{\mbf}[1]{\mathbf{#1}}

\renewcommand{\today}{\number \day \space \ifcase \month \or January\or%
  February\or March\or April\or May\or June\or July\or August\or%
  September\or October\or November\or December\fi \space \number \year}

\DeclareMathOperator{\inc}{inc}

\newcommand{\ileft}[2][{}]{\ensuremath{\ell_{#1}(#2)}}
\newcommand{\iright}[2][{}]{\ensuremath{\mathit{r}_{#1}(#2)}}

\newcommand{\set}[1]{\ensuremath{\left\{ #1 \right\}}}

\newcommand{\twoplustwo}{\ensuremath{\mbf{2}+\mbf{2}}}
\newcommand{\oneplusthree}{\ensuremath{\mbf{1}+\mbf{3}}}

%% file: intord-dim.bbl
\begin{thebibliography}{10}

\bibitem{bogart:canonical-intords}
{\sc Bogart, K.~P., Rabinovich, I., and Trotter, Jr., W.~T.}
\newblock A bound on the dimension of interval orders.
\newblock {\em J. Combin. Theory Ser. A 21}, 3 (1976), 319--328.

\bibitem{bosekOnlineVersionRabinovitch2012}
{\sc Bosek, B., Kloch, K., Krawczyk, T., and Micek, P.}
\newblock On-line version of {{Rabinovitch}} theorem for proper intervals.
\newblock {\em Discrete Math. 312}, 23 (Dec. 2012), 3426--3436.

\bibitem{boyadzhiyskaSimpleProofCharacterizing2018}
{\sc Boyadzhiyska, S., Isaak, G., and Trenk, A.~N.}
\newblock A simple proof characterizing interval orders with interval lengths
  between 1 and k.
\newblock {\em Involve J. Math. 11}, 5 (Apr. 2018), 893--900.

\bibitem{boyadzhiyskaIntervalOrdersTwo2019}
{\sc Boyadzhiyska, S., Isaak, G., and Trenk, A.~N.}
\newblock Interval orders with two interval lengths.
\newblock {\em Discrete Appl. Math. 267\/} (Aug. 2019), 52--63.

\bibitem{cerioliCountingIntervalLengths2011}
{\sc Cerioli, M.~R., Oliveira, F. d.~S., and Szwarcfiter, J.~L.}
\newblock On counting interval lengths of interval graphs.
\newblock {\em Discrete Appl. Math. 159}, 7 (Apr. 2011), 532--543.

\bibitem{fishburn:intords-char}
{\sc Fishburn, P.~C.}
\newblock Intransitive indifference with unequal indifference intervals.
\newblock {\em J. Math. Psych. 7\/} (1970), 144--149.

\bibitem{fishburnClassesIntervalGraphs1985}
{\sc Fishburn, P.~C., and Graham, R.~L.}
\newblock Classes of interval graphs under expanding length restrictions.
\newblock {\em J. Graph Theory 9}, 4 (1985), 459--472.

\bibitem{furedi:intord-shift}
{\sc F{\"u}redi, Z., Hajnal, P., R{\"o}dl, V., and Trotter, W.~T.}
\newblock Interval orders and shift graphs.
\newblock In {\em Sets, {{Graphs}} and {{Numbers}}}, A.~Hajnal and V.~T. Sos,
  Eds., vol.~60 of {\em Colloq. {{Math}}. {{Soc}}. {{J\'anos Bolyai}}}. 1991,
  pp.~297--313.

\bibitem{golumbicToleranceGraphs2004}
{\sc Golumbic, M.~C., and Trenk, A.~N.}
\newblock {\em Tolerance Graphs}, vol.~89 of {\em Cambridge {{Studies}} in
  {{Advanced Mathematics}}}.
\newblock {Cambridge University Press, Cambridge}, 2004.

\bibitem{hostenOrderDimensionComplete1999}
{\sc Ho{\c s}ten, S., and Morris, W.~D.}
\newblock The order dimension of the complete graph.
\newblock {\em Discrete Math. 201}, 1 (Apr. 1999), 133--139.

\bibitem{isaakBoundedDiscreteRepresentations1993}
{\sc Isaak, G.}
\newblock Bounded discrete representations of interval orders.
\newblock {\em Discrete Appl. Math. 44}, 1 (July 1993), 157--183.

\bibitem{joosCharacterizationMixedUnit2015}
{\sc Joos, F.}
\newblock A {{Characterization}} of {{Mixed Unit Interval Graphs}}.
\newblock {\em J. Graph Theory 79}, 4 (2015), 267--281.

\bibitem{kelly:3-dim-irr}
{\sc Kelly, D.}
\newblock The 3-irreducible partially ordered sets.
\newblock {\em Canad. J. Math. 29}, 2 (1977), 367--383.

\bibitem{kleitmanDedekindProblemNumber1975}
{\sc Kleitman, D., and Markowsky, G.}
\newblock On {{Dedekind}}'s {{Problem}}: {{The Number}} of {{Isotone Boolean
  Functions}}. {{II}}.
\newblock {\em Trans. Am. Math. Soc. 213\/} (1975), 373--390.

\bibitem{koblerIntervalGraphRepresentation2015}
{\sc K{\"o}bler, J., Kuhnert, S., and Watanabe, O.}
\newblock Interval graph representation with given interval and intersection
  lengths.
\newblock {\em J. Discrete Algorithms 34\/} (Sept. 2015), 108--117.

\bibitem{peerIntervalGraphsSide1995}
{\sc Pe'er, I., and Shamir, R.}
\newblock Interval graphs with side (and size) constraints.
\newblock In {\em Algorithms \textemdash{} {{ESA}} '95\/} ({Berlin,
  Heidelberg}, 1995), P.~Spirakis, Ed., Lecture {{Notes}} in {{Computer
  Science}}, {Springer}, pp.~142--154.

\bibitem{peerRealizingIntervalGraphs1997}
{\sc Pe'er, I., and Shamir, R.}
\newblock Realizing {{Interval Graphs}} with {{Size}} and {{Distance
  Constraints}}.
\newblock {\em SIAM J. Discrete Math. 10}, 4 (Nov. 1997), 662--687.

\bibitem{rabinovitchDimensionTheorySemiordersIntervalOrders1973}
{\sc Rabinovitch, I.}
\newblock {\em The {{Dimension}}-{{Theory}} of {{Semiorders}} and
  {{Interval}}-{{Orders}}}.
\newblock PhD thesis, Dartmouth College, 1973.

\bibitem{rabinovitch:semiorder-dim}
{\sc Rabinovitch, I.}
\newblock The dimension of semiorders.
\newblock {\em J. Combin. Theory Ser. A 25}, 1 (1978), 50--61.

\bibitem{rautenbachUnitIntervalGraphs2013}
{\sc Rautenbach, D., and Szwarcfiter, J.~L.}
\newblock Unit {{Interval Graphs}} of {{Open}} and {{Closed Intervals}}.
\newblock {\em J. Graph Theory 72}, 4 (2013), 418--429.

\bibitem{scott:semiorders}
{\sc Scott, D., and Suppes, P.}
\newblock Foundational aspects of theories of measurement.
\newblock {\em J. Symb. Logic 23\/} (1958), 113--128.

\bibitem{shuchatUnitIntervalOrders2016}
{\sc Shuchat, A., Shull, R., and Trenk, A.~N.}
\newblock Unit {{Interval Orders}} of {{Open}} and {{Closed Intervals}}.
\newblock {\em Order 33}, 1 (Mar. 2016), 85--99.

\bibitem{shuchatUnitMixedInterval2014}
{\sc Shuchat, A., Shull, R., Trenk, A.~N., and West, L.~C.}
\newblock Unit {{Mixed Interval Graphs}}.
\newblock {\em Congr. Numer. 221\/} (2014), 189--223.

\bibitem{trotter:dimbook}
{\sc Trotter, W.~T.}
\newblock {\em Combinatorics and Partially Ordered Sets: {{Dimension}} Theory}.
\newblock Johns {{Hopkins Series}} in the {{Mathematical Sciences}}. {Johns
  Hopkins University Press}, {Baltimore, MD}, 1992.

\bibitem{trotter:intords}
{\sc Trotter, W.~T.}
\newblock New perspectives on interval orders and interval graphs.
\newblock In {\em Surveys in Combinatorics, 1997 ({{London}})}, vol.~241 of
  {\em London {{Math}}. {{Soc}}. {{Lecture Note Ser}}.} {Cambridge Univ.
  Press}, {Cambridge}, 1997, pp.~237--286.

\bibitem{trotterPersonalCommunication2020}
{\sc Trotter, W.~T.}
\newblock Personal communication.
\newblock 2020.

\bibitem{trotter:3-dim-irr}
{\sc Trotter, Jr., W.~T., and Moore, Jr., J.~I.}
\newblock Characterization problems for graphs, partially ordered sets,
  lattices, and families of sets.
\newblock {\em Discrete Math. 16}, 4 (1976), 361--381.

\bibitem{Tro-Moo-77}
{\sc Trotter, Jr., W.~T., and Moore, Jr., J.~I.}
\newblock The dimension of planar posets.
\newblock {\em J. Combin. Theory Ser. B 22}, 1 (1977), 54--67.

\end{thebibliography}
